\documentclass[11pt,twoside]{article}

\newcommand{\ifims}[2]{#1}   % latex version
\newcommand{\ifAMS}[2]{#1}   % 1 - AMS, 2 - JEL
\newcommand{\ifau}[3]{#1}  % #1 - 1 author, #2 - 2 authors, #3 - 3 authors

\newcommand{\ifbook}[2]{#1}   % 1 - paper; 2 - book

\ifbook{

  }
  {

  }

\def\thetitle{A result on the bias of sieve profile estimators}
\def\theruntitle{A result on the bias of sieve profile estimators}

\def\theabstract{
This paper complements the paper \cite{AASP2013}. We show how to control the bias of a sieve type profile estimator under natural conditions on the Hessian of the expected contrast functional.
} 

\def\kwdp{62F10}
\def\kwds{62H12}

\def\thekeywords{semiparametric estimation, profile estimator, bias}

\def\authora{Andreas Andresen }
\def\runauthora{andreas andresen }
\def\addressa{
    \\ Weierstrass-Institute, \\
    Mohrenstr. 39, \\
    10117 Berlin, Germany     
    }
\def\emaila{andresen@wias-berlin.de}
\def\affiliationa{}%{Weierstrass-Institute}
\def\thanksa{The author is supported by Research Units 1735 
"Structural Inference in Statistics: Adaptation and Efficiency"
}

\date{}

\usepackage{amsmath,amssymb,amsthm}
\usepackage{natbib}
\usepackage{epsfig,graphicx}
\usepackage{comment}
\usepackage{color}
\usepackage{srcltx}
\usepackage[mathscr]{eucal}
\usepackage[math]{easyeqn}
\usepackage{etoolbox}

\ifims{
\textheight=23cm
\textwidth=14.8cm
\topmargin=0pt
\oddsidemargin=1.0cm
\evensidemargin=1.0cm
\linespread{1.3}
\renewenvironment{abstract}
    {\centerline{\textbf{Abstract}}\bigskip
      \begin{center}
       \begin{minipage}{11cm}
        \begin{small}
    }
    {   \end{small}
       \end{minipage}
      \end{center}
     \bigskip
    }

}{ %ims style
}

\numberwithin{equation}{section}
\numberwithin{figure}{section}
\newcounter{example}[section]
\numberwithin{example}{section}
\newcounter{remark}[section]
\numberwithin{remark}{section}
\newtheorem{theorem}{Theorem}[section]

\newtheorem{lemma}[theorem]{Lemma}

\newtheorem{exmp}[example]{Example}
\newtheorem{rmrk}[remark]{Remark}
\newenvironment{example}{\begin{exmp}\rm}{\end{exmp}}
\newenvironment{remark}{\begin{rmrk}\rm}{\end{rmrk}}

\begin{document}
\thispagestyle{empty}
\ifims{
\title{\thetitle}
\ifau{ % 1 author
  \author{
    \authora
    \ifdef{\thanksa}{\thanks{\thanksa}}{}
    \\[5.pt]
    \addressa \\
    \texttt{ \emaila}
  }
}
{  % 2 authors
  \author{
    \authora
    \ifdef{\thanksa}{\thanks{\thanksa}}{}
    \\[5.pt]
    \addressa \\
    \texttt{ \emaila}
    \and
    \authorb
    \ifdef{\thanksb}{\thanks{\thanksb}}{}
    \\[5.pt]
    \addressb \\
    \texttt{ \emailb}
  }
}
{   % 3 authors
  \author{
    \authora
    \ifdef{\thanksa}{\thanks{\thanksa}}{}
    \\[5.pt]
    \addressa \\
    \texttt{ \emaila}
    \and
    \authorb
    \ifdef{\thanksb}{\thanks{\thanksb}}{}
    \\[5.pt]
    \addressb \\
    \texttt{ \emailb}
    \and
    \authorc
    \ifdef{\thanksc}{\thanks{\thanksc}}{}
    \\[5.pt]
    \addressc \\
    \texttt{ \emailc}
  }
}

\maketitle
\pagestyle{myheadings}
\markboth
 {\hfill \textsc{ \small \theruntitle} \hfill}
 {\hfill
 \textsc{ \small
 \ifau{\runauthora}
      {\runauthora and \runauthorb}
      {\runauthora, \runauthorb, and \runauthorc}
 }
 \hfill}
\begin{abstract}
\theabstract
\end{abstract}

\ifAMS
    {\par\noindent\emph{AMS 2000 Subject Classification:} Primary \kwdp. Secondary \kwds}
    {\par\noindent\emph{JEL codes}: \kwdp}

} % end front latex
{ % front ims
\begin{frontmatter}
\title{\thetitle}
%\thankstext{T1}{\thankstitle}

% "Title of the paper"

\runtitle{\theruntitle}

\ifau{ % 1 author
\begin{aug}
    \author{\authora\ead[label=e1]{\emaila}}
    \address{\addressa \\
     \printead{e1}}
\end{aug}

 \runauthor{\runauthora}
\affiliation{\affiliationa} }
{ % 2 authors
\begin{aug}
    \author{\authora\ead[label=e1]{\emaila}\thanksref{t21}}
    \and
    \author{\authorb\ead[label=e2]{\emailb}\thanksref{t22}}
    
    \address{\addressa \\
     \printead{e1}}
    \address{\addressb \\
     \printead{e2}}
    \thankstext{t21}{\thanksa}
    \thankstext{t22}{\thanksb}
    \affiliation{\affiliationa, \affiliationb} 
    \runauthor{\runauthora and \runauthorb}
\end{aug}
} 
{ % 3 authors
\begin{aug}
    \author{\authora\ead[label=e1]{\emaila}\thanksref{t21}}
    \and
    \author{\authorb\ead[label=e2]{\emailb}\thanksref{t22}}
    \and
    \author{\authorc\ead[label=e3]{\emailc}\thanksref{t23}}
    
    \address{\addressa \\
     \printead{e1}}
    \address{\addressb \\
     \printead{e2}}
    \address{\addressc \\
     \printead{e3}}
    \thankstext{t21}{\thanksa}
    \thankstext{t22}{\thanksb}
    \thankstext{t23}{\thanksc}
    \affiliation{\affiliationa, \affiliationb, \affiliationc} 
    \runauthor{\runauthora, \runauthorb, and \runauthorc}
\end{aug}}

\begin{abstract}
\theabstract
\end{abstract}



\end{frontmatter}
} % end front ims

\def\ND{\cc{N}}
\def\Bernoulli{\mathrm{Bernoulli}}
\def\Vola{\mathrm{Vola}}
\def\Poisson{\mathrm{Poisson}}
\def\ag{\mathrm{ag}}
\def\glob{\operatorname{glob}}
\def\blk{\operatorname{block}}
\def\lin{\operatorname{lin}}
\def\cond{\, \big| \,}

\def\rdl{\epsilon}
\def\rd{\bb{\rdl}}
\def\rddelta{\delta}
\def\rdomega{\varrho}
\def\rddeltab{\rddelta^{*}}
\def\rhorb{\rhor^{*}}

\def\wv{\bb{w}}
\def\varthetav{\bb{\vartheta}}
\def\Lr{\breve{L}}
\def\zetavr{\breve{\zetav}}
\def\etavr{\breve{\etav}}
\def\xivr{\breve{\xiv}}

\def\rdb{\rd}
\def\rdm{\underline{\rdb}}

\def\taub{\tau_{\rdb}}
\def\taum{\tau_{\rdm}}
\def\kappab{\kappa_{\rd}}
\def\deltab{\delta_{\rd}}

\def\taubGP{\tau_{\rdb,\GP}}
\def\taumGP{\tau_{\rdm,\GP}}
\def\kappabGP{\kappa_{\rd,\GP}}
\def\deltabGP{\delta_{\rd,\GP}}
\def\nubm{\nu_{\rd}}
\def\uub{u_{\rd}}
\def\uubGP{u_{\rd,\GP}}
\def\nubmGP{\nu_{\rd, G}}

\def\rG{\rd,\GP}

\def\LinSp{\mathrm{L}}
\def\Id{I\!\!\!I}
\def\Ind{\operatorname{1}\hspace{-4.3pt}\operatorname{I}}

\def\BG{\mathcal{R}}
\def\bg{r}
\def\fmup{\phi}
\def\rg{r}
\def\uc{u_{c}}
\def\muc{\mu_{c}}
\def\mud{\mu_{0}}
\def\xxd{\xx_{0}}
\def\yyd{\yy_{0}}
\def\gmd{\gm_{0}}

\def\ms{m^{*}}
\def\Inv{A}
\def\InvT{\Inv^{\T}}
\def\Invt{\tilde{\Inv}}

\def\ssize{N}
\def\nsize{{n}}

\def\rhor{\omega}

\def\LT{L}
\def\LGP{\LT_{\GP}}
\def\La{\mathbb{L}}
\def\Lab{\La_{\rdb}}
\def\Lam{\La_{\rdm}}

\def\DP{D}
\def\DPc{\DP_{0}}
\def\DPb{\DP_{\rdb}}
\def\DPm{\DP_{\rdm}}

\def\LabGP{\La_{\rdb,\GP}}
\def\LamGP{\La_{\rdm,\GP}}

\def\DPbGP{\DP_{\rdb,\GP}}
\def\DPmGP{\DP_{\rdm,\GP}}
\def\riskbGP{\riskt_{\rdb,\GP}}

\def\gmi{\mathtt{b}}
\def\gmiid{\mathtt{g}_{1}}
\def\kullbi{\Bbbk}
\def\Thetasi{\Theta_{\loc}}
\def\rri{\mathtt{u}}
\def\rris{\rri_{0}}

\def\Ipc{\bb{\mathrm{f}}}
\def\IF{\Bbb{F}}
\def\IFc{\IF_{0}}
\def\IFb{\IF_{\rdb}}
\def\IFm{\IF_{\rdm}}

\def\DF{\cc{D}}
\def\DFc{\DF_{0}}
\def\DFb{\DF_{\rdb}}
\def\DFm{\breve{\DF}_{\rd}}
\def\DFm{\DF_{\rdm}}

\def\DPr{\breve{\DP}}
\def\VF{\cc{V}}
\def\VFc{\VF_{0}}

\def\HHc{\HH_{0}}
\def\HHb{\HH_{\rd}}
\def\HHm{\HH_{\rdm}}

\def\xib{\xi^{*}}
\def\xivb{\xiv_{\rdb}}
\def\xivm{\xiv_{\rdm}}
\def\CAm{\underline{\CA}}
\def\CAb{\CA}

\def\penr{\operatorname{pen}}
\def\pen{\mathfrak{t}}
\def\PEN{\operatorname{PEN}}
\def\RSS{\operatorname{RSS}}
\def\med{\operatorname{med}}

\def\ex{\mathrm{e}}
\def\entrl{\mathbb{Q}}
\def\entrlb{\entrl}
\def\entr{\entrl}

\def\kullb{\cc{K}} %{\wp}
\def\kullbc{\kullb^{c}}

\def\gm{\mathtt{g}}
\def\gmc{\gm_{c}}
\def\gmb{\gm}
\def\gmbm{\gmb_{1}}

\def\yy{\mathtt{y}}
\def\yyc{\yy_{c}}
\def\xx{\mathtt{x}}
\def\xxc{\xx_{c}}
\def\tc{t_{c}}

\def\alp{\alpha}
\def\alpn{\rho}
\def\gmu{\mathfrak{a}}

\def\losst{\varrho}
\def\loss{\wp}
\def\lossp{u}
\def\closs{g}

\def\riskt{\cc{R}}
\def\emprisk{\ell}
\def\bias{b}
\def\bern{q}

\def\TT{\nsize}

\def\Pone{P}
\def\Pf{\P_{f(\cdot)}}
\def\Ef{\E_{f(\cdot)}}
\def\Ps{\P_{\thetas}}
\def\Es{\E_{\thetas}}
\def\Pu{\P_{\upsilons}}
\def\Eu{\E_{\upsilons}}

\def\Pvs{\P_{\thetavs}}
\def\Evs{\E_{\thetavs}}

\def\UPd{w}
\def\nunup{\nu_{1}}
\def\rru{\rr_{1}}
\def\rups{\rr_{0}}
\def\rupsb{\rups^{*}}
\def\rrf{\rr^{\flat}}

\def\smooths{\mathbb{S}}
\def\smooth{\smooths_{1}}

\def\elli{\bar{\ell}}

\def\K{K}

\def\Psir{\breve{\Psi}}

\def\af{a}
\def\afs{\af^{*}}

\def\kapla{\varkappa}

\newcommand{\mlew}[1]{\tilde{\thetav}_{#1}}
\newcommand{\mlea}[1]{\hat{\thetav}_{#1}}
\newcommand{\mluw}[1]{\tilde{\theta}_{#1}}
\newcommand{\mlua}[1]{\hat{\theta}_{#1}}
\newcommand{\penm}[1]{\boldsymbol{m}_{#1}}

\def\Pdom{\mu_{0}}
\def\PDOM{\bb{\mu}_{0}}
\def\EDOM{\E_{0}}

\def\mk{m}
\def\Mk{\cc{M}}
\def\SV{\cc{S}}

\def\Cs{E}
\def\Csd{\Cs^{\circ}}
\def\Ca{A}
\def\CS{\cc{E}}
\def\CA{\cc{A}}
\def\CAb{\CA_{\rd}}
\def\CAC{\CA_{\CoFu}}

\def\Ccb{m_{\rdb}}
\def\Ccm{m_{\rdm}}
\def\CcbGP{m_{\rdb,\GP}}
\def\CcmGP{m_{\rdm,\GP}}

\def\etas{\eta^{*}}

\def\omegav{\bb{\phi}}
\def\omegavs{\omegav^{*}}
\def\omegavc{\omegav'}

\def\nuvs{\nuv^{*}}
\def\nuvc{\nuv'}

\def\nunu{\nu_{0}}
\def\numu{\nu_{1}}
\def\nupi{\nu^{+}}
\def\nubu{\beta}

\def\nus{\nu}
\def\nusb{\nus}
\def\nusr{\nus^{\bracketing}}
\def\Nusb{\mathbb{N}}
\def\Nusr{\mathbb{N}^{\diamond}}

\def\dist{d}
\def\distd{\mathfrak{a}}

\def\hatk{\kappa}
\def\ko{k^{\circ}}

\def\qqq{\mathfrak{q}}
\def\ppp{{s}}
\def\Cqq{C(\qqq)}
\def\Cqqb{C^{\diamond}(\qqq)}
\def\Crho{C(\mrho)}
\def\Cqqm{\log(4)}
\def\Cqpr{(\qqq \rrp + \dimp / 2)}

\def\Cdima{\mathfrak{e}_{0}}
\def\Cdimb{\mathfrak{e}_{1}}
\def\cdima{\mathfrak{c}_{0}}
\def\cdimb{\mathfrak{c}_{1}}
\def\cdim{\mathfrak{c}}

\def\rdomega{\varrho}
\def\deltaD{\delta}
\def\alphai{\alpha_{1}}
\def\alphaii{\alpha_{2}}
\def\alphaiii{\alpha_{3}}
\def\alphaiv{\alpha_{4}}

\def\err{\diamondsuit}
\def\errbm{\bar{\err}_{\rdomega}}
\def\errm{\err_{\rdm}}
\def\errb{\err_{\rdb}}

\def\errbGP{\err_{\rdomega,\GP}}
\def\errmGP{\err_{\rdm,\GP}}
\def\errbmGP{\bar{\err}_{\rd,\GP}}

\def\errs{\err_{\rdomega}^{*}}
\def\deltas{\alpha}

\def\xivbGP{\xiv_{\rdb,\GP}}
\def\xivmGP{\xiv_{\rdm,\GP}}

\def\SP{S}
\def\GP{G}
\def\GPt{\GP_{0}}
\def\GPn{\GP_{1}}
\def\gp{g}
\def\gs{s}

\def\errbGP{\err_{\rdb,\GP}}
\def\errmGP{\err_{\rdm,\GP}}
\def\errpmGP{\err_{\GP}^{\pm}}

\def\LCS{\cc{C}}

\def\DPGP{\DP_{\GP}}
\def\thetavsGP{\thetavs_{\GP}}

\def\LL{\cc{L}}
\def\LLb{\LL^{*}}
\def\LLh{\cc{L}}

\def\YY{\cc{Y}}
\def\LP{L^{\circ}}

\def\modcnrd{\mathfrak{A}}

\def\pens{\pi}
\def\pnn{\mathfrak{g}}
\def\pnnd{\mathfrak{u}}
\def\pnndGP{\pnnd_{\GP}}

\def\confpr{\mathfrak{c}}
\def\confprb{\confpr^{*}}

\def\pn{\pens^{*}}
\def\penInt{\mathfrak{D}}
\def\penH{\mathbb{H}}
\def\pmu{\mathfrak{u}}
\def\Closs{\cc{R}}

\def\dimp{p}
\def\riskb{\riskt_{\rdb}}
\def\dimpp{\dimp+1}
\def\BB{I\!\!B}
\def\vA{\mathtt{v}}

\def\deficiency{\Delta}
\def\spread{\Delta}
\def\dimtotal{\dimp^{*}}

\def\thetav{\bb{\theta}}
\def\thetavs{\thetav^{*}}
\def\thetavc{\thetav'}
\def\thetavd{\thetav^{\circ}}
\def\thetavdc{\thetav^{\sharp}}
\def\dthetavs{\thetav,\thetavs}

\def\thetas{\theta^{*}}
\def\thetac{\theta'}
\def\thetad{\theta^{\circ}}
\def\thetab{\theta^{\dag}}
\def\thetavb{\thetav^{\dag}}

\def\vtheta{\vartheta}
\def\vthetav{\bb{\vtheta}}
\def\prior{\Pi}

\def\Gam{\Xi}
\def\Gam{\mathcal{S}}
\def\RG{R}
\def\Psu{\Upsilon}
\def\Phim{\breve{\Phi}}

\def\Proj{P}

\def\gammavs{\gammav^{*}}
\def\gammavd{\gammav^{\circ}}
\def\etavs{\etav^{*}}
\def\etavd{\etav^{\circ}}
\def\etavc{\etav'}

\def\taus{\tau_{0}}
\def\taup{\lceil \tau \rceil}

\def\sigmas{{\sigma^{*}}}
\def\Sigmas{\Sigma_{0}}

\def\upsilonc{\upsilon'}
\def\upsilond{\upsilon^{\circ}}
\def\upsilonp{{\upsilon}^{*}}
\def\upsilonm{{\upsilon}_{*}}
\def\upsilonvs{\upsilonv^{*}}
\def\upsilons{\upsilon^{*}}
\def\upsilonb{\bar{\upsilon}}
\def\upsilonvd{\upsilonv^{\circ}}

\def\ups{\bb{\upsilon}}
\def\upss{\ups_{0}}
\def\upsc{\ups^{\prime}}
\def\upsd{\ups^{\circ}}
\def\upsdc{\ups^{\sharp}}
\def\upsdu{\ups^{\flat}}

\def\Ups{\varUpsilon}
\def\Upsd{\Ups^{\circ}}
\def\Upss{\Ups_{\circ}}
\def\UpsP{\Ups^{c}}

\def\Thetas{\Theta_{0}}
\def\ThetasGP{\Theta_{0,\GP}}
\def\varthetav{\bb{\vartheta}}

\def\glink{g}

\def\fvs{\fv}
\def\fs{f}
\def\fb{\fv^{\dag}}

\def\uc{\uv'}
\def\ud{\uv^{\circ}}
\def\uvs{\uv^{*}}
\def\us{u^{*}}
\def\vs{v^{*}}

\def\reps{\epsilon}
\def\eps{\epsilon}

\def\repsc{\reps_{0}}
\def\repsb{\reps^{*}}
\def\repsg{g}

\def\lu{\delta}
\def\lub{\bar{\lu}}

\def\Uu{U}
\def\UU{\cc{Y}}
\def\UUM{\cc{M}}
\def\UP{\cc{U}}
\def\up{\mathfrak{u}}

\def\VP{V}
\def\VPc{\VP_{0}}
\def\VPV{\cc{U}}
\def\VPVc{\cc{\VPV}_{0}}
\def\VPGP{\VP_{\GP}}
\def\VPSP{\VP_{\SP}}

\def\VV{H}
\def\GV{\cc{G}}
\def\GVS{S}

\def\VVb{\VV^{*}}
\def\VVc{\VV_{0}}
\def\vv{\bb{h}}
\def\vva{g}
\def\vp{\mathbf{v}}
\def\vpc{\vp_{0}}
\def\VVca{\VV}
\def\Vtt{H}

\def\DG{D}

\def\Vn{V_{0}}
\def\vn{v_{0}}

\def\norm{\mathfrak{c}}
\def\normc{\delta}
\def\norma{c}

\def\egridd{\cc{E}_{\delta}}
\def\penb{\varkappa}

\def\dotzeta{\dot{\zeta}}
\def\mes{\pi}
\def\mesl{\Lambda}
\def\cprr{F}

\def\lambdam{\gm_{1}}
\def\lambdaB{{\lambda}^{*}}
\def\lambdac{{\lambda'}}

\def\cla{{b}}
\def\fis{\mathfrak{a}}
\def\fiss{\fis_{1}}

\def\Vd{{V}}
\def\vd{\bar{v}}

\def\klim{k^{\circ}}
\def\midm{\mid \!}

\def\Ldrift{M}
\def\ldrift{m}
\def\mY{b}
\def\Lvar{D}
\def\lvar{\sigma}

\def\Mubcu{\Upsilon}
\def\Dthetav{\bb{u}}

\def\B{\cc{B}}
\def\BD{\B^{\circ}}
\def\BU{B}
\def\BI{\B^{*}}

\def\mub{\mu^{*}}
\def\mubc{\mu}
\def\mubcb{\mubc^{*}}
\def\Mubc{\mathbb{M}}
\def\Mubcb{\mathrm{M}}

\def\zzc{\zz_{c}}
\def\ww{w}
\def\wwc{\ww_{c}}

\def\norms{\circ} %{\vartriangle}
\def\rs{\rr_{\norms}}
\def\yys{\yy_{\norms}}
\def\xxs{\xx_{\norms}}
\def\zzs{\zz_{\norms}}
\def\uu{\mathtt{u}}
\def\uus{\uu_{\norms}}
\def\mus{\mu_{\norms}}
\def\gms{\gm_{\norms}}
\def\wws{\ww_{\circ}}

\def\srho{s}
\def\mrho{\varrho}

\def\Lmgf{\mathfrak{M}}
\def\Lmgfb{\Lmgf^{*}}

\def\lmgf{\mathfrak{m}}
\def\lmgfb{\lmgf^{*}}

\def\Expzeta{\mathfrak{N}}
\def\expzeta{\mathfrak{s}}

\def\rr{\mathtt{r}}
\def\rrb{\rr^{*}}
\def\rru{\rr_{\circ}}
\def\rrc{\rr'}
\def\rs{r_{*}}

\def\zz{\mathfrak{z}}
\def\zzb{\tilde{\zz}}
\def\tt{\mathfrak{t}}
\def\zb{z_{\rd}}
\def\zzg{\zz_{1}}
\def\zzQ{\zz_{0}}
\def\zzq{\zz}

\def\Cr{\mathfrak{c}}
\def\Crp{\mathfrak{C}}
\def\Crl{\mathfrak{r}}
\def\Crlp{\mathfrak{R}}
\def\Crlq{\cc{T}}
\def\Crlmu{\cc{M}}

%%%%%%%%%%%%   semipar   %%%%%%%%%%%%%%%%%%%%%%%%
\def\zetah{\zeta_{h}}
\def\GG{G}
\def\HH{H}
\def\pG{p}
\def\pH{q}
\def\hh{H^{*}}

\def\mubch{\mubc_{1}}
\def\rhoh{\rho_{1}}
\def\CoFuh{\CoFu_{1}}
\def\dimh{p_{1}}
\def\VPh{\VP_{1}}
\def\VPt{\VP_{0}}

\def\LLh{L_{1}}
\def\pnndh{\pnnd_{1}}

\def\LCS{C}
\def\Ac{A_{0}}
\def\Ab{A_{\rd}}
\def\DPrb{\DPr_{\rdb}}
\def\DPrm{\DPr_{\rdm}}
\def\Cb{\cc{C}_{\rdb}}
\def\Ub{\cc{U}_{\rdb}}
\def\zetavrb{\zetavr_{\rd}}
\def\xivrb{\breve{\xiv}_{\rd}}
\def\VPrb{\breve{\VP}_{\rdb}}
\def\Larb{\breve{\La}_{\rdb}}
\def\Larm{\breve{\La}_{\rdm}}

\def\deltav{\bb{\delta}}

\def\score{\nabla}
\def\scorer{\breve{\nabla}}

\def\LCS{C}
\def\Ac{A_{0}}
\def\Bc{B_{0}}
\def\AF{A}
\def\Ab{A_{\rdb}}
\def\Am{A_{\rdm}}
\def\DPrc{\DPr_{0}}
\def\DPrb{\DPr_{\rdb}}
\def\DPrm{\DPr_{\rdm}}
\def\Cb{\cc{C}_{\rdb}}
\def\Cm{\cc{C}_{\rdm}}
\def\Ub{\cc{U}_{\rdb}}
\def\deltav{\bb{\delta}}
\def\nuv{\bb{\nu}}
\def\xivrb{\breve{\xiv}_{\rd}}
\def\VPrb{\breve{\VP}_{\rdb}}
\def\Larb{\breve{\La}_{\rdb}}
\def\Lar{\breve{\La}}
\def\Larm{\breve{\La}_{\rdm}}
\def\VH{Q}
\def\VHc{\VH_{0}}
\def\zetavrm{\zetavr_{\rdm}}
\def\N{\mathbb{N}}

\def\Span{\operatorname{span}}
\def\Exc{{\square}}
\def\UUs{U_{\circ}}
\def\errbm{\errb^{*}}
\def\corrDF{\nu}
\def\BBr{\breve{\BB}}
\def\taua{\tau}
\def\AssId{\mathcal{I}}
\def\assId{\iota}
\def\AFD{\cc{A}}

\def\BanX{\cc{X}}
\def\basX{\ev}
\def\apprX{\alpha}
\def\fvs{\fv^{*}}
\def\lkh{\ell}
\def\Bc{B_{0}}
\def\dimn{\dimp_{\nsize}}
\def\betan{\beta_{\nsize}}

%%%%%%%%%%%   BvM   %%%%%%%%%%%%%%%%%%%%%%%%%%%%%%

\def\xivGP{\xiv_{\GP}}
\def\dimA{\mathtt{p}}
\def\dimAGP{\dimA}
\def\dime{\dimA_{e}}
\def\dimG{\dimA_{\GP}}
\def\dimS{\dimA_{s}}
\def\nubm{\nu_{\rd}}
\def\uub{u_{\rd}}
\def\uubGP{u_{\rd,\GP}}

\def\priorden{\pi}
\def\xivGP{\xiv_{\GP}}
\def\dimAGP{\dimA}
\def\nubm{\nu_{\rd}}
\def\uub{u_{\rd}}
\def\uubGP{u_{\rd,\GP}}

\def\CR{\mathcal{C}}
\def\CRb{\CR_{\rdb}}
\def\vthetavb{\bar{\vthetav}}
\def\Covpost{\mathfrak{S}}

\def\Db{\DP_{+}}
\def\Dm{\DP_{-}}
\def\uvb{\uv_{+}}
\def\uvm{\uv_{-}}
\def\uud{\omega}
\def\taub{\delta}
\def\Lip{L}
\def\Xb{X_{+}}
\def\Xm{X_{-}}
\def\deltam{\delta_{-}}
\def\betauv{\delta}
\def\betab{\betauv_{1}}
\def\betaf{\betauv_{2}}
\def\upsv{\bb{\varkappa}}
\def\upsvb{\bar{\upsv}}
\def\rhob{\varrho}
\def\alpb{\alp_{1}}
\def\betap{\betauv_{3}}
\def\Ec{\E^{\circ}}
\def\ff{f}
\def\fpos{g}
\def\fneg{h}
\def\alpb{\alp_{+}}
\def\alpm{\alp_{-}}

%%%%%%%%%%%   sms   %%%%%%%%%%%%%%%%%%%%%%%%%%%%%%
\def\kappak{\kappa}
\def\kappas{\kappak^{*}}
\def\Kappak{\cc{K}}
\def\DPk{\DP_{\kappak}}
\def\VPk{\VP_{\kappak}}

%%%%%%%%%%%%   sp   %%%%%%%%%%%%%%%%%%%%%%%%%%%%%
\def\ts{s}
\def\tsv{\bb{\ts}}
\def\mm{\kappa}
\def\mmc{\mm'}
\def\mmd{\mm^{\circ}}
\def\mmo{\mm^{*}}
\def\mmmmo{\mm,\mmo}
\def\mmt{\tilde{\mm}}
\def\mma{\hat{\mm}}
\def\pp{z}

\def\LLL{L_{1}}
\def\LLr{L_{0}}
\def\muL{\mu_{1}}
\def\mur{\mu_{0}}

\def\LmgfL{\Lmgf_{1}}
\def\Lmgfr{\Lmgf_{0}}
\def\Lmgfm{\Lmgf_{1}}

\def\Kappa{\cc{K}}
\def\CoFu{\cc{C}}
\def\CoFuc{\CoFu_{0}}
\def\CoFub{\CoFu^{*}}
\def\CoFuL{\CoFu_{1}}
\def\CoFur{\CoFu_{0}}
\def\CAL{\CA_{1}}
\def\CAr{\CA_{0}}
\def\CAzz{\cc{A}}

\def\pnnL{\pnn_{1}}
\def\pnnr{\pnn_{0}}
\def\ttd{\delta}
\def\alphaL{\alpha_{1}}
\def\alphar{\alpha_{0}}
\def\alpharL{\alpha}
\def\rat{\mathfrak{t}}
\def\mquad{\nquad}
\def\zzL{\zz_{1}}
\def\zzr{\zz_{0}}

\def\mmset{\mathcal{I}}
\def\xex{u}
\def\dcm{q}
\def\dc{g}
\def\dcL{\dc_{1}}
\def\dcr{\dc_{0}}
\def\kk{k}

\def\cpen{\tau}

%=================  density  ==============
\def\dens{f}
\def\jj{j}
\def\JJ{\cc{J}}
\def\Zphi{Z}
\def\Zphiv{\bb{\Zphi}}

%================= LES =====================
\def\nuu{\mathfrak{u}}
\def\nud{\mathfrak{u}_{0}}
\def\nun{c_{\nuu}}
\def\rhork{\kullb}
\def\GH{\mbox{GH}}
\def\HYP{\mbox{HYP}}
\def\NIG{\mbox{NIG}}
\def\IR{{\rm I\!R}}
\def\taggr{b}
\def\penm{\boldsymbol{m}}
\def\Crlp{\cc{R}}

%================== qfu/cmr ====================
\def\Mh{M}
\def\Mht{\Mh^{c}}

\def\Mhh{\Mh^{-}}
\def\Mhc{G}
\def\Lh{L_{1}}
\def\Uh{\cc{U}}
\def\wloc{w}
\def\Bias{B}
\def\bias{b}
\def\ExpzetaU{\Expzeta_{1}}
\def\vpci{\vp_{i,0}}
\def\IFci{\IF_{i,0}}

\def\erqb{\Circle_{\rdb}}
\def\erqm{\Circle_{\rdm}}
\def\errqm{\errm^{*}}
\def\errqb{\errb^{*}}
\def\Nsize{N}
\def\VVD{\VV_{1}}
\def\AA{A}
\def\Wloc{W}

\renewcommand{\(}{$\,}
\renewcommand{\)}{\,$}

\def\nquad{\hspace{-1cm}}
\def\eqdef{\stackrel{\operatorname{def}}{=}}
\def\tod{\stackrel{d}{\longrightarrow}}
\def\tow{\stackrel{w}{\longrightarrow}}
\def\toP{\stackrel{\P}{\longrightarrow}}

\newcommand{\cc}[1]{\mathscr{#1}}
\newcommand{\bb}[1]{\boldsymbol{#1}}

\renewcommand{\bar}[1]{\overline{#1}}
\renewcommand{\hat}[1]{\widehat{#1}}
\renewcommand{\tilde}[1]{\widetilde{#1}}

\renewcommand{\Gamma}{\varGamma}
\renewcommand{\Pi}{\varPi}
\renewcommand{\Sigma}{\varSigma}
\renewcommand{\Delta}{\varDelta}
\renewcommand{\Lambda}{\varLambda}
\renewcommand{\Psi}{\varPsi}
\renewcommand{\Phi}{\varPhi}
\renewcommand{\Theta}{\varTheta}
\renewcommand{\Omega}{\varOmega}
\renewcommand{\Xi}{\varXi}
\renewcommand{\Upsilon}{\varUpsilon}
\def\nn{\nonumber \\}

\def\suml{\sum\limits}
\def\supl{\sup\limits}
\def\maxl{\max\limits}
\def\infl{\inf\limits}
\def\intl{\int\limits}
\def\liml{\lim\limits}
\def\Cov{\operatorname{Cov}}
\def\Var{\operatorname{Var}}
\def\arginf{\operatornamewithlimits{arginf}}
\def\argsup{\operatornamewithlimits{argsup}}
\def\argmax{\operatornamewithlimits{argmax}}
\def\argmin{\operatornamewithlimits{argmin}}
\def\val{\operatorname{val}}

\def\D{\boldsymbol{D}}
\def\dd{\operatorname{d}}
\def\tr{\operatorname{tr}}
\def\I{I\!\!I}
\def\R{I\!\!R}
\def\E{I\!\!E}
\def\P{I\!\!P}
\def\X{\mathfrak{X}}
\def\kappa{\varkappa}
\def\Const{\mathrm{Const.} \,}
\def\cdt{\boldsymbol{\cdot}}
\def\tm{\!\times\!}
\def\T{\top}
\def\diag{\operatorname{diag}}
\def\diam{\operatorname{diam}}
\def\rank{\operatorname{rank}}
\def\loc{\operatorname{loc}}

\def\av{\bb{a}}
\def\bv{\bb{b}}
\def\cv{\bb{c}}
\def\dv{\bb{d}}
\def\ev{\bb{e}}
\def\fv{\bb{f}}
\def\gv{\bb{g}}
\def\hv{\bb{h}}
\def\iv{\bb{i}}
\def\jv{\bb{j}}
\def\kv{\bb{k}}
\def\lv{\bb{l}}
\def\mv{\bb{m}}
\def\nv{\bb{n}}
\def\ov{\bb{o}}
\def\pv{\bb{p}}
\def\qv{\bb{q}}
\def\rv{\bb{r}}
\def\sv{\bb{s}}
\def\tv{\bb{t}}
\def\uv{\bb{u}}
\def\vv{\bb{v}}
\def\wv{\bb{w}}
\def\xv{\bb{x}}
\def\yv{\bb{y}}
\def\zv{\bb{z}}

\def\Cv{\bb{C}}
\def\Gv{\bb{G}}
\def\Mv{\bb{M}}
\def\Sv{\bb{S}}
\def\Uv{\bb{U}}
\def\Xv{\bb{X}}
\def\Yv{\bb{Y}}
\def\Zv{\bb{Z}}

\def\alphav{\bb{\alpha}}
\def\epsv{\bb{\varepsilon}}
\def\etav{\bb{\eta}}
\def\gammav{\bb{\gamma}}
\def\varepsilonv{\bb{\varepsilon}}
\def\phiv{\bb{\phi}}
\def\psiv{\bb{\psi}}
\def\tauv{\bb{\tau}}
\def\upsilonv{\bb{\upsilon}}
\def\xiv{\bb{\xi}}
\def\zetav{\bb{\zeta}}

\def\Psiv{\bb{\Psi}}
\def\CONST{\mathtt{C}}

\def\itemv{\vfill\item}
\newenvironment{myslide}[1]
    {\begin{frame}\frametitle{#1}\vfill}
    {\vfill\end{frame}}

\def\vsp{\vspace{0.05\textheight} \vfill}
\def\summarysign{\resizebox{0.08\textwidth}{0.08\textheight}{\includegraphics{summary}}\,}
\def\nix{}
\def\wpu{$\bullet$}

\def\btri{\vfill{\( \blacktriangleright \) }}
\def\btrir{\vfill{\( \blacktriangleright \) }}

\newcommand{\mygraphics}[3]{\begin{center}
    \resizebox{#1\textwidth}{#2\textheight}{\includegraphics{#3}}
    \end{center}
}

\newcommand{\mybox}[3]{\begin{center}
    \resizebox{#1\textwidth}{#2\textheight}{#3}
    \end{center}
}

%\definecolor{myhcolor}{rgb}{0.2,0,0.8}
%\definecolor{myhcolor}{named}{red}
\newenvironment{eqnh}
{
    %\color{myhcolor}} {}
    \setbeamercolor{postit}{fg=black,bg=hellgelb} %{fg=myhcolor,bg=white}
    \begin{beamercolorbox}[center,wd=\textwidth]{postit} %rounded=true,shadow=true,
    \begin{eqnarray*}}
    {\end{eqnarray*}\end{beamercolorbox}
}

\def\gps{s}

\def\gps{s}
\def\GK{\cc{G}}
\def\Excgr{\diamondsuit}

\def\dimh{m}
\def\LCS{C}
\def\Bc{B_{0}}
\def\AF{A}
\def\CF{C}
\def\Ab{A_{\rdb}}
\def\Am{A_{\rdm}}
\def\DPrp{\DPr_{\dimh}}
\def\DPrb{\DPr_{\rdb}}
\def\DPrm{\DPr_{\rdm}}
\def\Cb{\cc{C}_{\rdb}}
\def\Cm{\cc{C}_{\rdm}}
\def\Ub{\cc{U}_{\rdb}}
\def\xivrb{\breve{\xiv}_{\rd}}
\def\VPrb{\breve{\VP}_{\rdb}}
\def\Larb{\breve{\La}_{\rdb}}
\def\Lar{\breve{\La}}
\def\Larm{\breve{\La}_{\rdm}}
\def\VH{Q}
\def\VHc{\VH_{0}}
\def\kappavrm{\kappavr_{\rdm}}

\def \Xv{\bb{X}}

\def\fvh{\bb{\dimh}}
\def\N{\mathbb{N}}
\def\Z{\mathbb{Z}}

\def\iic{\IF}
\def\iif{\breve{\iic}}
\def\DP{\text{D}}
\def\HH{\text{H}}
\def\A{\text{A}}
\def\ifc{\breve{\iic}}

\def\HF{\mathcal H}

\def\Thetathetav#1{\substack{\\[0.1pt] \upsilonv\in\Theta \\[1pt] \Proj \upsilonv = #1}}
\def\Span{\operatorname{span}}
\def\Exc{{\square}}
\def\UUs{U_{\circ}}
\def\errbm{\errb^{*}}
\def\corrDF{\rho}
\def\BBr{\breve{\BB}}
\def\taua{\tau}
\def\AssId{\mathcal{I}}
\def\AFD{\cc{A}}

\def\BanX{\cc{X}}
\def\basX{\ev}
\def\apprX{\alpha}
\def\fvs{\fv^{*}}
\def\lkh{\ell}
\def\Bc{B_{0}}
\def\h{\frac{1}{2}}
\def\basis{\ev}
\def\Proj{\Pi_{0}}
\def\Projh{\Pi_{\dimtotal}}

\def\Ij{\mathcal{I}}

\def\Mn{M_{\nsize}}
\def\bA{\breve{A}}
\def\cA{\bA_{\dimh}}

\def\Sdr{\cc{S}}
\def\xxn{\xx_{\nsize}}

\def\CONST{\mathtt{C}}
\def\Ij{\mathcal{I}}
\def\etas{\eta^{*}}
\def\kappavs{\kappav^{*}}
\def\kappavc{\kappav'}

\def\kappav{\boldsymbol{\kappa}}
\def\kappavs{\kappav^{*}}

\def\omegav{\bb{\phi}}
\def\omegavs{\omegav^{*}}
\def\omegavc{\omegav'}

\def\dimn{\dimp_{\nsize}}
\def\betan{\beta_{\nsize}}

\def\bA{\breve{A}}
\def\cA{\bA_{\dimh}}

\def\corrDF{\rho}
\def\upsilonv{\boldsymbol{\upsilon}}
\def\upsilonvs{\boldsymbol{\upsilon}^*}
\def\upsilonvd{\boldsymbol{\upsilon}^\circ}

\section{Introduction}
This paper presents a way to control the bias in a sieve profile contrast estimation problem, which we elaborate for simplicity for parameters in \(l^2\eqdef \{(x_1,x_2,\ldots)\subset \R:\, \sum_{k=1}^\infty x^2<\infty\}\). More precisely consider a contrast functional \(\E\LL:\R^{\dimp}\times l^2\to \R\). Assume that the goal is to calculate the target parameter \( \thetavs \) defined as
\begin{EQA}[c]
    \thetavs
    \eqdef
    \argmax_{\thetav} \sup_{\etav: (\thetav,\etav) \in \Ups \subset l^2} \E \LL(\thetav,\etav) =\Pi_{\thetav}\upsilonvs\eqdef\Pi_{\thetav} \argmax_{\upsilonv=(\thetav,\etav) \in \Ups\subset  l^2} \E \LL(\upsilonv),
\label{thetavsBanX}
\end{EQA}
with a set \(\Ups\subset \R^{\dimp}\times l^2\). 
To circumvent the problem of maximizing over an infinite dimensional set \(\Ups\subset\R^{\dimp}\times  l^2\) define for some \(\dimh\in\N\) the following approximation contrast functional
\begin{EQA}
\E\LL_{\dimh}:\R^{\dimp+\dimh}&\to& \R,\\
(\thetav,\etav)&\mapsto& \E\LL(\thetav,E_{l^2}\etav),
\end{EQA}
where \(E_{l^2}: \R^{\dimh} \to l^2\) is the natural embedding operator. Further define the biased target
\begin{EQA}[c]
    \thetavs_{\dimh}
    \eqdef
    \argmax_{\thetav} \sup_{\etav \in \R^{\dimh}} \E \LL_{\dimh}(\thetav,\etav)
    =\Pi_{\thetav}\upsilonvs_{\dimh} \eqdef\Pi_{\thetav} \argmax_{\upsilonv=(\thetav,\etav) \in \R^{\dimtotal}} \E \LL_{\dimh}(\thetav,\etav).
\label{thetavsBanX}
\end{EQA}
We are interested in a bound for the euclidean distance \(\|\thetavs-\thetavs_{\dimh}\|\).
Further consider 
%\( \DP^{2} \), \( \HH^{2} \), and \( \A \) be 
the following block representations of the hessian operator
\( \DF^{2}(\upsilonv) = - \nabla^{2} \E \LL(\thetav,\fv) \):
\begin{EQA}
    \label{eq: introduction of operator representations upsilonv}
    \DF^{2}&=&
     \left(
      \begin{array}{ccc}
        \DF_{\dimh}^2 & \AF_{\dimh} \\
        \AF_{\dimh}^\T & \HF^2_{\dimh} 
      \end{array}
    \right)\in L(\R^{\dimtotal}\times l^2,\R^{\dimtotal}\times l^2),
\label{DFc012se}
\end{EQA}
where for a vectorspace \(V\) the symbol \(L(V,V)\) denotes the set of linear operators from \(V\)  to \(V\), \(\dimtotal=\dimp+\dimh\in\N\) and where
\begin{EQA}[c]
\DF^2_{\dimh}(\upsilonv)\eqdef\nabla_{\dimh}^2\E\LL(\upsilonv)\in\R^{\dimtotal\times\dimtotal},
\end{EQA}
i.e. the derivatives of \(\E\LL\) are only taken with respect to the first \(\dimp+\dimh\in\N\) coordinates of \(\upsilonv=(\thetav,\etav)\in \R^{\dimp}\times l^2\)

Define the following two matrices
\begin{EQA}
\DPr^{2}_{\dimh}(\upsilonv)\eqdef\left(\Pi_{\thetav}\DF_{\dimh}^{-2}(\upsilonv)\Pi_{\thetav}^\T\right)^{-1}\in\R^{\dimp\times\dimp},
&\text{ 
and }&
\DPr^{2}(\upsilonv)\eqdef \left(\Pi_{\thetav}\DF^{-2}(\upsilonv)\Pi_{\thetav}^\T\right)^{-1}\in\R^{\dimp\times\dimp}.
\end{EQA}
The second result we want to derive is a bound for the difference between \(\DPrp(\upsilonvs_{\dimh})\in\R^{\dimp\times\dimp}\) and \(\DPr(\upsilonvs)\in\R^{\dimp\times\dimp}\) in spectral norm, where 
\begin{EQA}[c]
\upsilonvs_{\dimh}\eqdef \argmax_{\upsilonv\in\R^{\dimp+\dimh}}\E\LL.
\end{EQA}

This kind of problem arises for instance when a sieve profile estimator is analyzed as in \cite{AASP2013}. Given a (random) contrast functional \(\LL:\R^{\dimp}\times l^2\to \R\) one defines \(\LL_{\dimh}\) analogously to \(\E\LL_{\dimh}\) above and the sieve profile estimator
\begin{EQA}[c]
\label{ttdhttdhs}
\tilde \thetav_{\dimh}\eqdef \argmax_{\thetav\in\R^{\dimp}}\max_{\etav\in\R^{\dimh}}\LL_{\dimh}(\thetav,\etav)\eqdef \Pi_{\thetav}\upsilonvs_{\dimh}\eqdef \argmax_{\upsilonv\in\R^{\dimp+\dimh}}\E\LL.
\end{EQA}
The parametric results obtained in \cite{AASP2013} claim that the profile estimator
\( \tilde{\thetav}_{\dimh} \) estimates well \( \thetavs_{\dimh} \) if the spread \(\Excgr(\rr,\xx)>0\) is small. More precisely we have for fixed \(\xx \) with Theorem 2.1 of \cite{AASP2013} applied to 
\( \tilde{\thetav}_{\dimh} \) from \eqref{ttdhttdhs} that with probability greater \(1-22.8 \ex^{-\xx}\)
\begin{EQA}[c]
\label{eq: convergence of estimator for dimh}
    \|\DPrp \bigl( \tilde{\thetav}_{\dimh} - \thetavs_{\dimh} \bigr)
    - \xivr_{\dimh}(\upsilonvs_{\dimh})\|\le \Excgr\big(\rr,\xx\big),
\label{DPrctthsh}
\end{EQA} 
where \( \upsilonvs_{\dimh} = (\thetavs_{\dimh},\etavs_{\dimh}) 
= \argmax_{\upsilonv} \E \LL_{\dimh}(\upsilonv) \).

This result involves exactly the two kinds of bias from above, i.e. one that concerns the difference \(\thetavs_{\dimh} - \thetavs\) and the other the difference between \(\DPrp\in\R^{\dimp\times\dimp}\) and \(\DPr\in\R^{\dimp\times\dimp}\). 

\cite{AASP2013} use two assumptions to address this bias. The first one reads:

\begin{description}
\item[\(\bb{(bias)}\)] There exists a  decreasing function \(\alpha:\N\to \R_+\) such that
\begin{EQA}[c]
\|\DPrp(\thetavs_{\dimh} - \thetavs) \|\le \hat\alpha(\dimh).
\end{EQA}
\end{description}

The second one:

\begin{description}
\item[\(\bb{(bias')}\)] As \(\dimh\to \infty\)
\begin{EQA}
\| I-\DPrp(\upsilonvs)^{-1}\DPr(\upsilonvs)^2\DPrp(\upsilonvs)^{-1}\|&=& o(1),\\
\| I-\DPrp(\upsilonvs_{\dimh})^{-1}\DPrp(\upsilonvs)^2\DPrp(\upsilonvs_{\dimh})^{-1}\|&=& o(1).%\\
% \Cov\left((\frac{1}{\sqrt n}\DPr_{\dimh})^{-1}(\score_{\thetav} \lkh_{1}(\upsilonvs_{\dimh}) - \A_{\dimh}\HH_{\dimh}^{-2}\score_{\etav} \lkh_{1}(\upsilonvs_{\dimh}))\right)&\to& 1.
 \end{EQA}
\end{description}
In this paper we want to present a particular way to obtain such a function \(\hat\alpha(\dimh)\) and to derive \(\bb{(bias')}\), which relies on less high level conditions on the smoothness and structure of \(\E\LL\).

\section{Main result}
Denote by \(\Pi_{\dimtotal}:l^2\to \R^{\dimtotal}\) the projection to the first \(\dimtotal\in\N\) coordinates of an element of \(l^2\).
To bound the bias \(\|\DPrp(\thetavs_{\dimh} - \thetavs) \|>0\) we present the following condition:

\begin{description}
  \item[\( \bb{(\kappav)} \)] 
  The vector \( \kappavs\eqdef (Id_{l^2}-\Pi_{\dimtotal})\upsilonvs\) satisfies \(\|\HF_{\dimh}\kappavs\|^2\le \CONST_{\kappavs}\dimh\) for some \(\CONST_{\kappavs}>0\) and with \( \alpha(\dimh)\to 0 \)
\begin{EQA}
    \|\DF_{\dimh}^{-1} \AF_{\dimh} \kappavs\|
    &\le& 
    \alpha(\dimh). 
\label{apprXdimhledimh}
\end{EQA} 
Further for any \(\lambda\in[0,1]\) with some \(\tau(\dimh)\to 0\)
\begin{EQA}
\label{eq: condition on smoothness of DF}
    \|\DF_{\dimh}^{-1} \left(\nabla_{\upsilonv\kappav}\E\LL(\upsilonvs,\lambda\kappavs) -\AF_{\dimh}\right) \kappavs\|
    &\le& 
    \tau(\dimh),\\
\left|{\kappavs}^\T(\DF_{\kappav\kappav}-\nabla_{\kappav\kappav}\E\LL(\upsilonvs,\lambda\kappavs))\kappavs\right| &\le& \CONST_{\kappavs}^2\dimh.
\label{eq: condition on smoothness of DF-zeta}
\label{eq: new condition for bias}
\end{EQA}     
\end{description}

To ensure that \( \DPr_{\dimh} \) is close to \( \DPr \) we impose the following second condition.

\begin{description}
    \item[\( \bb{(\upsilonv\kappav)} \)]
Assume that with some \(\beta(\dimh)\to 0\)
\begin{EQA}[c]
\label{eq: convergence of breve matrix}
\|\HF^{-1}\AF^{\T}\DF_{\dimh}^{-1}\|\le \beta(\dimh).
\end{EQA}
\end{description}

\begin{theorem}
\label{theo: bias cond}
Let the conditions \(\bb{ (\kappav)} \) and \( \bb{(\upsilonv\kappav)} \) be 
fulfilled. Further let the condition \( \bb{(\cc{L}{\rr}_{\infty})}\) from Section \ref{sec: conditions} be satisfied for \( \E\LL: l^2\to \R \) with \(\gmi(\rr)\equiv \gmi>0\). Set \( {\rr^*}^{2}= 4\CONST_{\kappavs}^2\dimh/\gmi\) and let for some \(\dimh_0\in\N\) and all \(\dimh\ge \dimh_0\) the condition \( \bb{(\LL_{0})} \) be fulfilled for \(\DFc=\DF_{\dimh}\) and for any \(\rr\le \rr^*\). Then \(\bb{(bias)}\) and \(\bb{(bias')}\) are satisfied with
 \begin{EQA}   
   \hat\alpha(\dimh)&=&  \sqrt{\frac{1+\corrDF^2}{1-\corrDF^2}}\Bigg(\alpha(\dimh) +\tau(\dimh)+2\delta(2\rr^*)\rr^*\Bigg).%\\
  % \hat\beta(\dimh)&=& \frac{1+\corrDF^2+\beta^2(\dimh)}{1-\corrDF^2}\frac{\beta^2(\dimh)}{1-\beta^2(\dimh)}.
\end{EQA}
\end{theorem}

\section{Application to single-index model}
\label{sec: application to single index model}
We present an example to illustrate how these results can be derived for single-index modeling. 
Consider the following model
\begin{EQA}[c]
\label{eq: single index model introduced}
    \Yv_{i}
    =
    \fs(\Xv_{i}^{\T} \thetavs) + \varepsilon_{i}, 
    \qquad 
    i=1,...,\nsize,
\end{EQA}
or some \(\fs:\R\to \R\) and \(\thetavs\in S_{1}^{\dimp,+}\subset\R^{\dimp}\), i.i.d errors \(\varepsilon_i\in\R\) and \(\Var(\varepsilon_i)=\sigma^2\) and i.i.d random variables \(\Xv_{i}\in \R^{\dimp}\) with distribution denoted by \(\P^{\Xv}\). 

To ensure identifyability of \(\thetavs\in\R^{\dimp}\) we assume that it lies in the half sphere \(S_{1}^{\dimp,+}:=\{\thetav\in\R^{\dimp}:\, \|\thetav\|=1,\, \theta_1> 0\}\subset\R^{\dimp}\). We assume that the support of the \(\Xv_{i}\in \R^{\dimp}\) is contained in the ball of radius \(s_{\Xv}>0\). 
Further we assume that \(\fs\in \{f:[-s_{\Xv},s_{\Xv}]\mapsto \R\} \) can be well approximated by a orthonormal \(C^2\)-Daubechies-wavelet basis, i.e. for a suitable function \(\basX_{0}:=\psi:[-s_{\Xv},s_{\Xv}]\mapsto \R\) we set for \(k=2^{r_k}+j_k\)
\begin{EQA}[c]
\basX_{k}(t)=2^{k/2}\psi\left(2^{k}(t-2j_ks_{\Xv})\right),\, k\in\N.
\end{EQA}

Our aim is to analyze the properties of the profile MLE 
\begin{EQA}[c]
    \tilde{\thetav}
    \eqdef
    \argmax_{\thetav} \max_{\etav \in \R^{\dimh}} \LL(\thetav,\etav) ,
\label{ttSI}
\end{EQA}    
where
\begin{EQA}[c]
    \LL(\thetav,\etav)
    =
    - \frac{1}{2} 
    \sum_{i=1}^{\nsize} | \Yv_{i} - \sum_{k=0}^{m} \etav_{k} \basX_{k}(\Xv_{i}^{\T} \thetav) |^{2}.
\end{EQA}

Consider the following assumptions.

\begin{description} 
 \item[\((\mathbf{Cond}_{\Xv})\)]
  The measure \(\P^{\Xv}\) is absolutely continuous with respect to the Lebesgue measure. The Lebeque density \(d_{\Xv}\) of \(\P^{\Xv}\) is only positive on the ball \(B_{s_{\xv}}(0)\subset \R^{\dimp}\) and Lipshitz continuous with Lipshitz constant \(L_{d_{\Xv}}>0\).
  %\begin{comment}
  %Further for any orthonormal basis \((\mathbf b_k)_{k=1}^{\dimp}\subset \R^{\dimp}\) we assume that 
  %\begin{EQA}[c]
  %f_{\Xv}(\Xv)=\prod_{k=1}^{\dimp}f_{\Xv^\T\mathbf b_k}(\Xv^\T\mathbf b_k),
  %\end{EQA}
  %where \(f_{\Xv^\T\mathbf b_k}\) is a probability density on \([-s_{\Xv},s_{\Xv}]\).
  %\end{comment}
\end{description}

Of course we need some regularity of the link function \(\fs\in \{f:[-s_{\Xv},s_{\Xv}]\mapsto \R\} \):
\begin{description} 
  
  \item[\( (\mathbf{Cond}_{\fvs}) \)]
   For some \(\fvs\in\R^{\N}\)
   \begin{EQA}[c]
   \label{eq: expansion of indexfunciton}
    \fs=\fs_{\fvs}=\sum_{k=1}^\infty f^*_k\basX_{k},
   \end{EQA}
   where with some \(\alpha>2\) and a constant \(C_{\|\fvs\|}>0\)
   \begin{EQA}[c]\label{eq: smoothness of fs}
    \sum_{l=0}^\infty l^{2\alpha}{f^*_{l}}^2 \le C_{\|\fvs\|}^2< \infty.
   \end{EQA} 
\end{description}

\begin{lemma}
\label{lem: conditions theta eta}
Assume \( (\mathbf{Cond}_{\fvs}) \) and \((\mathbf{Cond}_{\Xv})\). Using our orthogonal and sufficiently smooth wavelet basis we get for any \(\lambda\in[0,1]\)
\begin{EQA}
 \|\DF_{\kappav\kappav}^{1/2}\kappavs\|^2<\CONST_1 n\dimh^{-2\alpha}, &&
\alpha(\dimh)\le \CONST_2\dimh^{-\alpha-1/2}\sqrt n,\\
\beta(\dimh)\le \CONST_3 \dimh^{-1/2},&&
 \tau(\dimh)\le \CONST_2\dimh^{-2\alpha+1}\sqrt{n},
\end{EQA}
and \(\left|{\kappavs}^\T(\DF_{\kappav\kappav}-\nabla_{\kappav\kappav}\E\LL(\upsilonvs,\lambda\kappavs))\kappavs\right|=0\).
\end{lemma}

For details see \cite{Andresensingleindex}.

\appendix
\section{Appendix}

\subsection{The conditions}
\label{sec: conditions}
We adopt the conditions from Section~3 of \cite{SP2011} with some 
minor changes. First we present the parametric conditions that apply to parametric models with finite dimensional parameter. Then explain two new conditions that arise in the infinite dimensional setting.

For some finite dimension \(\dimtotal\in\N\) the parametric conditions involve a matrix \( \DF_{0}^{2} \) and a central point \(\upsilonvd\in\R^{\dimtotal}\) that have to be specified before the conditions can be checked. 

\begin{remark}
For Theorem \ref{theo: bias cond} the matrix equals 
\begin{EQA}[c]
\DFc^{2}
    =
    - \nabla_{\dimh}^{2} \E \LL(\upsilonvs),
\end{EQA}
and \(\upsilonvd=\upsilonvs_{\dimh}\), i.e. the central point does not coincide with the element that defines the matrix \(\DFc^{2}\). It is important to note that condition \( \bb{(\LL_{0})} \) thus becomes another constraint on the bias.
\end{remark}

The matrix \( \DFc^{2} \) has to satisfy certain regularity conditions.
We begin by representing the information matrix in block form:
\begin{EQA}
    \DF_{0}^2
    &=&
    \left( 
      \begin{array}{cc}
        \DP_{0}^{2} & A_{0} \\
        A_{0}^{\T} & \HH_{0}^{2} \\
      \end{array}  
    \right).
\label{DFcseg0}
\end{EQA}
Here we restate \emph{identifiability conditions}:

\begin{description}
  \item[\( (\bb{\AssId}) \)] 
   It holds
\begin{EQA}[c]
    \| \HH_{0}^{-1} A_{0}^{\T} \DP_{0}^{-1} \|^2_{\infty}
    =:\corrDF< 1 .
\label{regularity2}
\end{EQA}
\end{description}

Using the matrix \(\DFc\in\R^{\dimtotal\times\dimtotal}\) and the central point \(\upsilonvd\in\R^{\dimtotal}\) we define the local set \(\Theta(\rr)\) with some
\(\rr\ge 0\) 
\begin{EQA}[c]
\Theta(\rr)\eqdef \{\upsilonv=(\thetav,\etav)\in\Theta,\, \|\DFc(\upsilonv-\upsilonvd)\|\le \rr\}.
\end{EQA}
The local conditions only describe the properties of the process \( \LL(\upsilonv) \) 
for \( \upsilonv \in \Theta(\rr) \) with some fixed value \( \rr>0 \). The global 
conditions have to be fulfilled on the whole \( \Theta \). 
We start with the local conditions.

\begin{description}
    \item[\( \bb{(\LL_{0})} \)]
    \textit{
%    and \( \fis > 0 \) such that \( \DP \ge \fis \VP \).
    For each \( \rr \le \rups \), 
    there is a constant \( \rddelta(\rr) \) such that
    it holds on the set \( \Upsilon(\rr) \) and with spectral norm \(\|\cdot\|\):
    }
\begin{EQA}[c]
\label{LmgfquadELGP}
    \bigl\|
       \DF_{0}^{-1} \nabla^{2}\E\LL(\upsilonv) \DF_{0}^{-1} - \Id_{\dimtotal} 
    \bigr\|
    \le
    \rddelta(\rr).
\end{EQA}

\end{description}

We also need:
\begin{description}
  \item[\(\bb{(\cc{L}{\rr}_{\infty})}\)]   For any \( \rr > \rups\) there exists a value \( \gmi(\rr) > 0 \), 
     such that
\begin{EQA}[c]
    \frac{-\E \LL(\upsilonv,\upsilonvs)}{\|\DF(\upsilonv-\upsilonvs)\|^{2}}
    \ge 
    \gmi\left(\|\DF(\upsilonv-\upsilonvs)\|\right).
\end{EQA}
\end{description}

\subsection{Proof of Theorem~\ref{theo: bias cond}}
%\subsection{Control of the bias}
 
\begin{lemma} 
\label{lem: apriori bias in sieve approach}
Assume that \( \bb{(\cc{L}{\rr}_{\infty})} \) is satisfied with \( \gmi(\rr) \equiv \gmi \) and that the condition \( \bb{(\kappav)} \) is satisfied. Then we get \(\|\DF(\upsilonvs_{\dimh}-\upsilonvs)\|\le \rr^*\) where \( {\rr^*}^{2}= 4\CONST_{\kappavs}\dimh/ \gmi \).
\end{lemma}

\begin{proof}
Note that
\begin{EQA}[c]
 \|\DF(\upsilonvs-\Pi_{\dimtotal}\upsilonvs)\|=\|\HF_{\dimh} \kappavs\|,
\end{EQA}
such that \(\upsilonvs\in\Theta(\rr^*)\). Further we have \(\nabla\E\LL(\upsilonvs)=0\) such that by the Taylor expansion with some \(\lambda\in[0,1]\) 
\begin{EQA}
\E\LL(\Pi_{\dimtotal}\upsilonvs,\upsilonvs)&=&\|\HF_{\dimh} \kappavs\|^2+{\kappavs}^\T(\HF_{\dimh}-\nabla_{\kappa\kappa}\E\LL(\upsilonvs,\lambda\kappavs))\kappavs.
\end{EQA}
which gives with \eqref{eq: condition on smoothness of DF-zeta} and \(\bb{(\kappav)}\) on \(\Theta(\rr^*)\) that
\begin{EQA}
\label{eq: bound for logratio projected oracle}
|\E\LL(\Pi_{\dimtotal}\upsilonvs,\upsilonvs)|&\le& \|\DF(\upsilonvs-\upsilonvs)\|^2 +\CONST_{\kappavs}\dimh\le 2\CONST_{\kappavs}\dimh.
\end{EQA}
Now we show that \( \upsilonvs_{\dimh} \) also belongs to \( \Theta(\rr^*) \) for \( {\rr^*}^{2} \ge 4\CONST_{\kappavs}\dimh/\gmi \). Suppose for the moment that
\( \bigl\| \DF (\upsilonvs_{\dimh} - \upsilonvs) \bigr\| > \rr^* \).
By \( \bb{(\cc{L}{\rr}_{\infty})} \), it holds 
\begin{EQA}[c]
    2 \bigl| \E \LL(\upsilonvs_{\dimh},\upsilonvs) \bigr| 
    \ge 
    \gmi \bigl\| \DF (\upsilonvs_{\dimh} - \upsilonvs) \bigr\|^{2} 
    >
    \gmi {\rr^*}^{2} .
\label{2ELL1s2}
\end{EQA}
This contradicts \( \bigl| \E \LL(\upsilonvs_{\dimh},\upsilonvs) \bigr| 
\le \bigl| \E \LL(\Pi_{\dimtotal}\upsilonvs,\upsilonvs) \bigr| \)
in view of \( {\rr^*}^{2} \ge 4\CONST_{\kappavs}\dimh/ \gmi  \) and \eqref{eq: bound for logratio projected oracle}, so \( \upsilonvs_{\dimh} \in \Theta(\rr^*) \).
\end{proof}

\begin{lemma} 
\label{lem: bias in sieve approach}
Assume that \( \bb{(\cc{L}{\rr}_{\infty})} \) is satisfied with \( \gmi(\rr) \equiv \gmi \). Further assume \( \bb{(\kappav)} \) and \( \bb{(\cc{L}_{0})} \) with central point \(\upsilonvs_{\dimh}\in\R^{\dimtotal}\) and operator \(\DF_{\dimh}\).
%\( \nsize \| \kappav \|^{2} \le C \dimtotal \), and \( \rr^{2} \ge C \dimtotal \).
Then we get with \( {\rr^*}^{2}= 4\CONST_{\kappavs}\dimh/ \gmi\)
\begin{EQA}[c]
\label{eq: bounds for bias}
 \|\breve\DP_{\dimh}(\thetavs_{\dimh}-\thetavs)\|\vee\|\DF_{\dimh}(\upsilonvs_{\dimh}-\upsilonvs)\|\le\sqrt{\frac{1+\corrDF^2}{1-\corrDF^2}} \Big(\alpha(\dimh) +\tau(\dimh)+2\delta(2\rr^*)\rr^*\Big).
\end{EQA}
%and \(\|\DF(\upsilonvs_{\dimh}-\omegavs)\|\le \sqrt{\frac{1+\corrDF^2}{1-\corrDF^2}} \Big(\alpha(\dimh) +2\delta(2\rr^*)\rr°*\Big)+\|\HF_{\dimh} \kappavs\|\).
\end{lemma}

\begin{proof}
Using condition \( \bb{(\cc{L}_{0})} \) and Taylor expansion we have on \({\Theta }_{\dimh}(\rr)=\{\|\DF_{\dimh}(\upsilonv-\upsilonvs_{\dimh})\|\le\rr\}\subset\R^{\dimp+\dimh}\)
\begin{EQA}
\label{eq: bound for expected value of approx error}
&&\nquad\sup_{\upsilonv\in\Theta(\rr)}\|\DF_{\dimh}^{-1}\nabla_{\dimh} \E\LL_{\dimh}(\upsilonv) - \DF_{\dimh}^{-1}\nabla_{\dimh} \E\LL_{\dimh}(\upsilonvs)-\DF_{\dimh} \, (\upsilonv - \upsilonvs)\|\\
  &\le& \sup_{\upsilonv\in\Theta(\rr)}\|\DF_{\dimh}^{-1}\nabla_{\dimh}^2 \E\LL_{\dimh}(\upsilonv)^2\DF_{\dimh}^{-1}-I_{\dimtotal}\|_{\DF\Upss^-(\rr),\DF\Upss^-(\rr)}\rr\\
  &\le& \delta(\rr)\rr.
\end{EQA}
Because of Lemma \ref{lem: apriori bias in sieve approach} we know that
\begin{EQA}
\|\DF_{\dimh}(\Pi_{\dimtotal}\upsilonvs-\upsilonvs_{\dimh})\|&=&\|\DF E_{l^2}(\Pi_{\dimtotal}\upsilonvs-\upsilonvs_{\dimh})\|\\
	&\le&\|\DF(\Pi_{\dimtotal}\upsilonvs-\omegavs)\|+\|\DF(\upsilonvs-E_{l^2} \upsilonvs_{\dimh})\|
	\le 2\rr^*,
\end{EQA}
such that \(\Pi_{\dimtotal}\upsilonvs_{\dimh}\in{\Theta}_{0,\dimh}(2\rr^*)\), which gives
\begin{EQA}[c]
\big\|\DF_{\dimh}\, (\upsilonvs_{\dimh} -\Pi_{\dimtotal}\upsilonvs)- \DF_{\dimh}^{-1} \nabla_{\dimp+\dimh} \E\Big(\LL(\Pi_{\dimtotal}\upsilonvs)-\LL(\upsilonvs_{\dimh})\Big)
	\big\|\le 2\delta(2\rr^*)\rr^*,
\end{EQA}
from which we derive with the triangle inequality
\begin{EQA}
\big\|\DF_{\dimh}\, (\upsilonvs_{\dimh} - \upsilonvs)\big\|
	 &\le& 2\delta(2\rr^*)\rr^*+\Big\|\DF_{\dimh}^{-1} \nabla_{\dimp+\dimh} \E\big(\LL(\Pi_{\dimtotal}\upsilonvs)-\LL(\upsilonvs_{\dimh})\big)\Big\|.
\end{EQA}
Because \(\nabla_{\dimp+\dimh} \E\LL(\upsilonvs_{\dimh})=0\) and \(\nabla \E\LL(\upsilonvs)=0\) we find
\begin{EQA}
&&\nquad\Big\|\DF_{\dimh}^{-1} \nabla_{\dimp+\dimh} \E\big(\LL(\Pi_{\dimtotal}\upsilonvs)-\LL(\upsilonvs_{\dimh})\big)\Big\|=\Big\|\DF_{\dimh}^{-1} \Pi_{\dimp+\dimh} \nabla\E\big(\LL(\Pi_{\dimtotal}\upsilonvs)-\LL(\upsilonvs)\big)\Big\|
\end{EQA}
Using that \(\|\DF(\Pi\upsilonvs-\upsilonvs)\|\le\rr^*\) and condition \(\bb{(\kappav)} \) we may infer by the Taylor expansion that with some \(\lambda\in[0,1]\)
\begin{EQA}
&&\nquad\Big\|\DF_{\dimh}^{-1} \Pi_{\dimp+\dimh} \nabla\E\big(\LL(\Pi_{\dimtotal}\upsilonvs)-\LL(\upsilonvs)\big)\Big\|\\
	&\le&\Big\|\DF_{\dimh}^{-1}\AF_{\dimh}\big(E_{l^2}\Pi_{\dimtotal}\upsilonvs-\upsilonvs\big)\Big\|+\Big\|\DF_{\dimh}^{-1} (\nabla_{\upsilonv\kappav}\E[\LL(\Pi_{\dimtotal}\upsilonvs))]- \AF_{\dimh}) \kappavs\Big\| \\
	&=& \Big\|\DF_{\dimh}^{-1} \AF_{\dimh}\kappavs\Big\|+\Big\|\DF_{\dimh}^{-1} \left(\nabla_{\upsilonv\kappav}\E[\LL\big((\Pi_{\dimtotal}\upsilonvs,\lambda\kappavs)\big)] -\AF_{\dimh}\right) \kappavs\Big\|.
\end{EQA}
Due to assumption  \( \bb{(\kappav)} \) the last sum is bounded by \(\alpha(\dimh)+\tau(\dimh) \). Together this gives that
\begin{EQA}[c]
\big\|\DF_{\dimh}\, (\upsilonvs_{\dimh} - \upsilonvs)\big\|=\alpha(\dimh) + \tau(\dimh)+2\delta(2\rr^*)\rr^*.
\end{EQA}
Finally we can represent
\begin{EQA}
\DF_{\dimh}^2= \left(
      \begin{array}{ccc}
        \DP^2 & \A\\
        \A^\T & \HH^2
      \end{array}
    \right), && \breve\DP_{\dimh}^2=\DP^2-\A^\T\HH^{-2}\A.
\end{EQA}
and due to \( (\bb{\AssId}) \) this gives
\begin{EQA}
 \|\breve\DP_{\dimh}(\thetavs_{\dimh}-\thetavs)\|^2&\le&\frac{1+\corrDF}{1-\corrDF}\|\DF_{\dimh}(\upsilonvs_{\dimh} - \upsilonvs)\|^2.
\end{EQA}

\begin{comment}
It remains to estimate
\begin{EQA}[c]
\Big\|\DF_{\dimh}^{-1} \Pi_{\dimp+\dimh}\DF\Big\|^2\eqdef \sup_{\mathbf w\in\R^{\dimtotal}}\frac{\langle \mathbf w,\DF_{\dimh}^{-1} \Pi_{\dimp+\dimh}\DF^2 \Pi_{\dimp+\dimh}^\T\DF_{\dimh}^{-1}\mathbf w \rangle }{\|\mathbf w\|^2} =1,
\end{EQA}
which completes the proof.
\end{comment}
\end{proof}

\begin{lemma}
 \label{lem: convergence of breve matrix}
Assume \( \bb{(\upsilonv\kappav)} \) then
\begin{EQA}[c]
\|I-\DPr_{\dimh}^{-1} \DPr^2\DPr_{\dimh}^{-1} \|\le\frac{1+\corrDF^2+\beta^2(\dimh)}{1-\corrDF^2}\frac{\beta^2(\dimh)}{1-\beta^2(\dimh)}.
\end{EQA}
\end{lemma}

\begin{proof}
Take any \(\mathbf v\in \R^{\dimp}\) with \(\|\mathbf v\|\le 1\) and note that with \(\upsilonv=(\thetav,\etav,\kappav)\in  l^2\)
\begin{EQA}
&&\nquad\DPr^{-2}\DPr_{\dimh}\mathbf v\\
&=&\Pi_{\thetav}\argmax_{\upsilonv\in l^2}\left\{\thetav^\T \DPr_{\dimh}\mathbf v-\| \DF \upsilonv\|^2/2\right\}\\
&=&\Pi_{\thetav}\argmax_{\upsilonv\in \R^{\dimp+\dimh}}\left\{\thetav^\T \DPr_{\dimh}\mathbf v-\| \DF_{\dimh} \upsilonv\|^2/2-\inf_{\kappav}(\kappav^\T\AF_{\dimh}^{\T}\upsilonv+\| \HF_{\dimh} \kappav\|^2/2)\right\}\\
&=&\Pi_{\thetav}\argmax_{\upsilonv\in \R^{\dimp+\dimh}}\left\{\thetav^\T \DPr_{\dimh}\mathbf v-\| \DF_{\upsilonv\upsilonv} \upsilonv\|^2/2-\|\DF_{\kappav\kappav}^{-1/2}\AF_{\dimh}^{\T}\upsilonv \|^2/2\right\}\\
&\eqdef&\Pi_{\thetav}\argmax_{\upsilonv\in \R^{\dimp+\dimh}}g(\upsilonv)\eqdef \Pi_{\thetav}\upsilonvd.
\end{EQA}
Setting the gradient of \(g(\cdot)\) equal to zero gives that the maximizer \(\upsilonvd\in \R^{\dimp+\dimh}\) satisfies
\begin{EQA}[c]
\upsilonvd=\DF_{\dimh}^{-1}(I_{\dimtotal}-\DF_{\dimh}^{-1}\AF_{\dimh}\HF_{\dimh}^{-2}\AF_{\dimh}^{\T}
\DF_{\dimh}^{-1})^{-1}\DF_{\dimh}^{-1}\Pi_{\thetav}^\T\DPr_{\dimh}\mathbf v,
\end{EQA}
where \(\Pi_{\thetav}^\T:\,\R^{\dimp}\to \R^{\dimp+\dimh}\) denotes the canonical embedding of \(\R^{\dimp}\) into \(\R^{\dimp+\dimh}\). By assumption we have
\begin{EQA}[c]
\|(I_{\dimtotal}-\DF_{\dimh}^{-1}\AF_{\dimh}\HF_{\dimh}^{-2}\AF_{\dimh}^{\T}\DF_{\dimh}^{-1})^{-1} -I_{\dimtotal}\|\le \frac{\beta^2(\dimh)}{1-\beta^2(\dimh)}.
\end{EQA}
Note that \(\DPr_{\dimh}\Pi_{\thetav}\DF_{\dimh}^{-2}\Pi_{\thetav}^\T\DPr_{\dimh}\mathbf v= \mathbf v\) which gives
\begin{EQA}
&&\nquad\|(I-\DPr_{\dimh} \DPr^{-2}\DPr_{\dimh})\mathbf v \|=\|\mathbf v-\DPr_{\dimh}\Pi_{\thetav}\upsilonvd\|=\|\DPr_{\dimh}\Pi_{\thetav}\DF_{\dimh}^{-2}\Pi_{\thetav}^\T\DPr_{\dimh}\mathbf v-\DPr_{\dimh}\Pi_{\thetav}\upsilonvd\|\\
&=&\|\DPr_{\dimh}\Pi_{\thetav}\DF_{\dimh}^{-1}\bigg((I_{\dimtotal}-\DF_{\dimh}^{-1}
\AF_{\dimh}\HF_{\dimh}^{-2}\AF_{\dimh}^{\T}\DF_{\dimh}^{-1})^{-1}-I_{\dimtotal}\bigg)\DF_{\dimh}^{-2}\Pi_{\thetav}^\T\DPr_{\dimh}\mathbf v\|\\
&\le&  \frac{\beta^2(\dimh)}{1-\beta^2(\dimh)}\|\DPr_{\dimh}\Pi_{\thetav}\DF_{\dimh}^{-2}\Pi_{\thetav}^\T\DPr_{\dimh}\mathbf v\|\\
  &=& \frac{\beta^2(\dimh)}{1-\beta^2(\dimh)}\| \mathbf v\|= \frac{\beta^2(\dimh)}{1-\beta^2(\dimh)}.
\end{EQA}
This implies
\begin{EQA}[c]
\|I-\DPr_{\dimh}^{-1} \DPr^2\DPr_{\dimh}^{-1} \|\le \|I-\DPr_{\dimh} \DPr^{-2}\DPr_{\dimh}\|\|\DPr_{\dimh}^{-1} \DPr^2\DPr_{\dimh}^{-1}\|\le \frac{1+\corrDF^2+\beta^2(\dimh)}{1-\corrDF^2}\frac{\beta^2(\dimh)}{1-\beta^2(\dimh)}.
\end{EQA}

\end{proof}

\begin{lemma}
 \label{lem: convergence of breve matrix}
Assume \( \bb{(\upsilonv\kappav)} \) then we get with \( {\rr^*}^{2}= 4\CONST_{\kappavs}\dimh/ \gmi\)
\begin{EQA}[c]
\| I-\DPrp(\upsilonvs_{\dimh})^{-1}\DPrp(\upsilonvs)^2\DPrp(\upsilonvs_{\dimh})^{-1}\|\le \delta(\rr^*)/(1-2\delta(\rr^*)).
 \end{EQA}
\end{lemma}

\begin{proof}
Denote \(\DPr_{\dimh\dimh}\eqdef\DPrp(\upsilonvs_{\dimh})\), \(\DF_{\dimh\dimh}\eqdef\DF_{\dimh}(\upsilonvs_{\dimh})\) and \(\DPrp\eqdef\DPrp(\upsilonvs)\), \(\DF_{\dimh}\eqdef\DF_{\dimh}(\upsilonvs)\). 
We simply calculate
\begin{EQA}
\|I-\DPr_{\dimh\dimh}^{-1} \DPrp^2\DPr_{\dimh\dimh}^{-1} \|&\le& \|\DPr_{\dimh\dimh}\DPrp^{-2}\DPr_{\dimh\dimh}-I_{\dimp}\|\|\DPr_{\dimh\dimh}^{-1} \DPrp^2\DPr_{\dimh\dimh}^{-1}\|.
\end{EQA}
Now we get with condition \( \bb{(\LL_{0})} \) and Lemma \ref{lem: apriori bias in sieve approach}
\begin{EQA}
 \|\DPr_{\dimh\dimh}\DPrp^{-2}\DPr_{\dimh\dimh}-I_{\dimp}\|& =&\|\DPr_{\dimh\dimh}\left(\Pi_{\thetav} \DF_{\dimh}^{-2}\Pi_{\thetav}^\T-\Pi_{\thetav}\DF_{\dimh\dimh}^{-2}\Pi_{\thetav}^\T\right)\DPr_{\dimh\dimh}\|\\
 &=&\|\DPr_{\dimh\dimh}\left(\Pi_{\thetav}\DF_{\dimh\dimh}^{-1}\left\{I-\DF_{\dimh\dimh}\DF_{\dimh}^{-2}\DF_{\dimh\dimh} \right)\DF_{\dimh\dimh}^{-1}\Pi_{\thetav}^\T\right)\DPr_{\dimh\dimh}\|\\
 &\le&\|I-\DF_{\dimh\dimh}\DF_{\dimh}^{-2}\DF_{\dimh\dimh}\|\\
 &\le&\|I-\DF_{\dimh\dimh}^{-1}\DF_{\dimh}^{2}\DF_{\dimh\dimh}^{-1}\|\|\DF_{\dimh\dimh}\DF_{\dimh}^{-2}\DF_{\dimh\dimh}\|\\
 &\le& \delta(\rr^*)\|\DF_{\dimh\dimh}\DF_{\dimh}^{-2}\DF_{\dimh\dimh}\|.
\end{EQA}
This implies
\begin{EQA}
\|\DF_{\dimh\dimh}\DF_{\dimh}^{-2}\DF_{\dimh\dimh}\|\le1/(1-\delta(\rr^*)),&& \|\DPr_{\dimh\dimh}^{-1} \DPrp^2\DPr_{\dimh\dimh}^{-1}\|\le (1-\delta(\rr^*))/(1-2\delta(\rr^*)).
\end{EQA}
Together this gives the claim.
\end{proof}

\bibliography{semiquellen}
  \end{document}